\documentclass[10pt,twoside,a4paper]{amsart}

\usepackage{amsfonts,amsmath,amsthm}
\usepackage{hyperref}

\newtheorem{thm}{Theorem}[section]

\newtheorem{lem}[thm]{Lemma}
\newtheorem{pro}[thm]{Proposition}

\theoremstyle{remark}
\newtheorem{remark}[thm]{Remark}

\newcommand{\R}{\mathbb{R}}
\newcommand{\Rd}{\mathbb{R}^{d}_{+}}

\newcommand{\Rq}{R_{\theta_{1}, \theta_{2}}^{q}}
\newcommand{\sken}{S_{R}^{+}}
\newcommand{\sbat}{S_{R}^{1}}
\newcommand{\sbi}{S_{R}^{2}}
\newcommand{\Rken}{\mathbb{R}^{d}_{+}(\alpha_{-})}
\newcommand{\Rgehi}{\mathbb{R}^{d}_{+}(\alpha_{+})}
\newcommand{\bken}{B_{R}^{+}}

\numberwithin{equation}{section}

\title[Robin problems]{The inverse Robin boundary value problem in a half-space}

\author[L. P\"aiv\"arinta, M. Zubeldia]{Lassi P\"aiv\"arinta and Miren Zubeldia}

\address{Department of Mathematics and Statistics, University of Helsinki. P.O. Box 68 (Gustaf H\"allstr\"ominkatu 2b) FI-00014}
\email{miren.zubeldia@helsinki.fi}
\email{lassi.paivarinta@helsinki.fi}

\date{\today}

\subjclass[2010]{35G15, 35J10, 35R30, 35Q60}

\keywords{inverse problems, Schr\"odinger equation, impedance parameter, Robin-to-Robin map}

\begin{document}

\maketitle

\begin{abstract}
We study the inverse Robin problem for the Schr\"odinger equation in a half-space. The potential is assumed to be compactly supported. We first solve the direct problem for dimensions two and three. We then show that the Robin-to-Robin map uniquely determines the potential $q$. 
\end{abstract}

\section{Introduction}

In this paper we consider a Robin boundary problem for wave propagation modelled by the Schr\"odinger equation
\begin{equation}\label{schro}
(\Delta + q(x))u(x)=0
\end{equation}
in the half-space $\R^{3}_{+} = \{(x_1, x_2, x_3) \in \R^3 : x_3 > 0 \}$ and in the half-plane $\R^2_{+}= \{ (x_1, x_2)\in \R^2 : x_2 > 0 \}$, with the impedance boundary condition
\begin{equation}\label{impedance}
\frac{\partial u}{\partial x_d} + \theta u = f \quad \quad \text{over} \quad \quad x_d=0,
\end{equation}
for $d=2,3$ and $\theta\in \mathbb{C}$, subject to suitable radiation conditions in each case. This problem arises in the propagation of electromagnetic and acoustic waves. 

We assume $\Im q(x) \geq 0$ in $\Rd$, $d=2,3$ and $\Im q(x) >0$ in an open and bounded domain $\Omega \subset \Rd$. In addition, it is required that $q(x)$ differs from the constant $k^{2}$ only in a bounded domain that contains $\Omega$, where the wave number $k$ is positive and fixed. It is customary to write the impedance parameter $\theta$ as $\theta = ik\beta$ where $\beta$ is a given constant, called an acoustic admittance. 

For the applications, $\beta$ would have to be taken as a complex-valued function in order to model the physical properties of the boundaries considered. Positive values of the real part $\Re \beta >0$ are associated with energy absorbing boundaries and those problems have been widely studied by Chandler-Wilde \cite{cw0} and Chandler-Wilde and Peplow \cite{cwp0}. If $\Re \beta =0$, the boundary is said to be passive \cite{Ka} or non-absorbing. Energy-absorbing boundaries allow the propagation of evanescent waves having an exponential decay from a source point at the boundary and in any direction. In the critical situation when $\Re \beta =0$ and $\Im \beta <0$, induced surface waves propagate being guided by the boundary and their exponential decay is strictly reduced to the cross-sectional direction (similar to a Rayleigh wave). So this is the critical case that maximizes the energy that propagates at boundary level, which is of interest for many applications in acoustics. See \cite{DMN2}, \cite{DMN1}, \cite{DMN} and the references given there.

On the one hand, we study the case when $\Re \beta =0$ and $\Im \beta < 0$, so that $\theta > 0$. These hypotheses on the impedance coefficients imply the appearance of surface waves guided by the domains infinite boundary. The Helmholtz equation $\Delta u + k^2 u = 0$ with this boundary condition is studied by Dur\'an, Muga, N\'ed\'elec \cite{DMN2}, \cite{DMN}. Looking for outgoing solutions, the authors firstly observe that the speed of an outgoing surface wave is different from the speed of the volume waves. Thus, a usual Sommerfeld radiation condition does not describe this phenomenon. In order to exhibit the correct radiation condition, they compute the associated Green's function and analyze its far field. 

In addition, we assume either $\beta = 0$ (rigid boundary) or $\Re \beta >0$ (energy-absorbing boundary). This means that either $\theta = 0$ or $\theta \in\mathbb{C}\backslash \{0\}$ such that $\Im \theta > 0$. In applications in outdoor sound propagation, $\beta$ lies only within a restricted region of the physically feasible half-plane $\Re \beta >0$. One of the basic theoretical problems arising from a study of outdoor sound propagation is the problem of predicting the far field behavior of the sound field emitted from a monofrequency monopole point source located above a homogeneous impedance plane \cite{cs1}, \cite{cs2}, \cite{hf}, \cite{khn}, \cite{r}, \cite{Th}. In mathematical terms, the solution for this problem is the Green function for Helmholtz equation in a half-space with impedance boundary condition. See for example \cite{cw3}, \cite{cw4}, \cite{cw1}, \cite{Ha}, \cite{Ra} among others. Thus a precise evaluation of the Green's function is of great importance. Accurate and efficient calculation of propagation from a line source for the widest range of source and receiver positions as well as a generalized asymptotic expansion for the Green's function in the far field is obtained by Chandler-Wilde and Hothershall \cite{cw1}, \cite{cw2}. This problem is two-dimensional. Propagation from a point source, i.e., three-dimensional case, has been considered by Thomasson \cite{Th} among others and has been studied in more detail by Chandler-Wilde \cite{cw}.

It turns out that in the papers mentioned above, the wave number $k$ is a positive constant. The complex wave propagation phenomena involving infinite half-planes, or perturbation of them, having impedance boundary conditions is studied by Dur\'an, Hein and N\'ed\'elec \cite{DHN}. Many problems that appear in sciences and engineering motivated their study. See for example \cite{cwp0}, \cite{ht}, \cite{ma}, \cite{sa}. Desiring to treat these problems, they study a classical two-dimensional model of time-harmonic waves, based on the Helmholtz equation. As the Helmholtz equation permits to describe a wide range of different wave propagation phenomena when the oscillations act in the linear range, the understanding of these phenomena for a relatively simple model allows to study them later for more complex and specific cases.

The first goal of this paper is to study the direct problem for the Schr\"odinger equation (\ref{schro}) with the impedance boundary condition (\ref{impedance}) for $d=2,3$. We assume that $\theta \geq 0$ or $\Im \theta > 0$. We base on the works \cite{DMN2}, \cite{DMN1}-\cite{DMN}, \cite{Th}, \cite{cw1}-\cite{cw2} mentioned above. We construct the unique solution of the Robin boundary problem
\begin{equation}\label{1}
\left\{ \begin{array}{ll}
(\Delta + q(x))u= 0 & \text{in $\Rd$}\\
\\
\frac{\partial u}{\partial x_{d}} + \theta u = f & \text{over $\{x_{d} = 0\}$}\\
\\
u \quad \text{satisfies the corresponding radiation condition}.
\end{array} \right.
\end{equation}
for $d=2,3$ in terms of the Green's function $G_{\theta}$ associated to the Helmholtz equation $\Delta u + k^2 u = 0$ with impedance boundary condition (\ref{impedance}).

Observe that when $\theta = 0$, (\ref{1}) becomes into a Neumann problem. Thus proving existence and uniqueness of solution of the problem (\ref{1}) solves the direct Neumann boundary problem for the Schr\"odinger equation in half-space. As far as we know, this does not appear in the literature.

The radiation condition that we refer in the problem (\ref{1}) is established from the one that the corresponding Green's function satisfies. We see that in the case of a non-absorbing boundary, the established radiation condition is somewhat different from the usual Sommerfeld's one, due to the appearance of surface waves. In the rigid or energy-absorbing boundary, we have the usual one. See Appendix (Section \ref{appendix}) below.

We assume that the boundary data $f$ is in the weighted Sobolev space $H^{-1/2, -\delta}(\R^{d-1})$ with $\delta > 1/2$ defined by 
\begin{equation}
H^{-1/2, -\delta}(\R^{d-1}) = \left\{ u: \frac{u}{(1+|x|^{2})^{\delta/2}} \in H^{-1/2}(\R^{d-1}) \right\}.
\end{equation}
Here we use the definition of the Sobolev space in terms of the Fourier transfrom
\begin{equation}\label{sobolev}
H^{s}(\R^d) = \left\{ f\in L^2(\R^d): \int_{\R^d} |\hat{f}(\xi)|^2 (1+|\xi|^2)^s \, d\xi < \infty \right\}, \quad \quad s\in \R.
\end{equation}
Thus it follows that
\begin{equation}\label{inclusion}
H^{1/2, -\delta}(\R^{d-1}) \subset H^{-1/2, -\delta}(\R^{d-1}).
\end{equation}

We show that the problem (\ref{1}) has a unique solution $u$ belonging to the weighted Sobolev space $H^{1, -\delta}(\Rd)$ with $\delta > 1/2$. Thus using the trace theorem that characterizes the existence of the trace operators 
\begin{align}
& \gamma_0 : H^1(\Rd) \to H^{1/2}(\R^{d-1}), \, \quad \quad \quad \, \gamma_0 v = v|_{x_d = 0}\notag\\
& \gamma_1 : H^1(\Rd) \to H^{-1/2}(\R^{d-1}), \quad \quad \quad \gamma_1 v= \frac{\partial v}{\partial x_d}\Big{|}_{x_d=0}\notag
\end{align}
and the definition of the weighted Sobolev spaces as above, the trace of $u$ and its normal derivative can be defined and it holds that
\begin{align}
 u|_{x_d=0} &\in H^{1/2, -\delta}(\R^{d-1}), \notag\\
 \frac{\partial u}{\partial x_d}\Big{|}_{x_d =0} &\in H^{-1/2, -\delta}(\R^{d-1}).\notag
\end{align}
Hence, we define the set of the Cauchy data for solutions of the Schr\"odinger equation in half-space as follows,
\begin{equation}\label{cauchy}
C_{q}:= \left\{ \left(u|_{x_{d}=0}, \frac{\partial u}{\partial x_{d}}\Big{|}_{x_{d}=0}\right) : u\in H^{1,-\delta}(\Rd) \, \, \, \delta>\frac{1}{2}, \, \, \,(\Delta + q(x))u = 0 \, \, \text{in} \, \, \Rd \right\}.
\end{equation}

We then introduce the Dirichlet-to-Neumann map $\Lambda_{q}$, the Neumann-to-Dirichlet map $N_{q}$ and the Robin-to-Robin map $R_{\theta_{1}, \theta_{2}}^{q}$ associated with $(\Delta + q(x))$ on $\Rd$ as follows,

\begin{equation}\label{dnmap}
\Lambda_{q}: \left\{ \begin{array}{lcr}
H^{1/2, -\delta}(\R^{d-1}) & \longrightarrow & H^{-1/2, -\delta}(\R^{d-1}) \\
u|_{x_{d}=0} & \longmapsto &  \frac{\partial u}{\partial x_{d}}\Big{|}_{x_{d}=0},
\end{array} \right.
\end{equation}

\begin{equation}\label{ndmap}
N_{q}: \left\{ \begin{array}{lcr}
R(\Lambda_{q}) \subset H^{-1/2, -\delta}(\R^{d-1}) & \longrightarrow & H^{1/2, -\delta}(\R^{d-1}) \\
\frac{\partial u}{\partial x_{d}}\Big{|}_{x_{d}=0} & \longmapsto & u|_{x_{d}=0},
\end{array} \right.
\end{equation}
where $R(\Lambda_q)$ denotes the range of $\Lambda_q$,

\begin{equation}\label{rrmap}
\Rq: \left\{ \begin{array}{lcr}
H^{-1/2, -\delta}(\R^{d-1}) & \longrightarrow & H^{-1/2, -\delta}(\R^{d-1}) \\
\frac{\partial u}{\partial x_{d}} + \theta_{1}u\Big{|}_{x_{d}=0} & \longmapsto & \frac{\partial u}{\partial x_{d}} + \theta_{2}u\Big{|}_{x_{d}=0},
\end{array} \right.
\end{equation}
where $\theta_{1}\neq \theta_{2}$. 


The inverse Robin boundary value problem is to determine $q(x)$ in the upper half-space from knowledge of the Robin-to-Robin map $R_{\theta_1, \theta_2}^q$. As far as we know, this problem is not studied for general $\theta_1$, $\theta_2$ in the half-space. The case when $\theta_1 =0$ and $\theta_2 = \infty$ is considered by Cheney, Lassas and Uhlmann \cite{clu} in dimension $d=3$. The inverse scattering problem for the acoustic equation with both $\theta = \cot \alpha < 0$ (active boundary) and $\theta= \cot \alpha >0$ (passive boundary) and for exponentially decaying potentials $q$ with $\Im \theta > 0$, has been considered by Karamyan in \cite{Ka1}, \cite{Ka} and \cite{Ka2}, respectively. The impedance boundary problem in open bounded domain for the particular case of $\theta_1 =-\tan\alpha$, $\theta_2=\cot\alpha$ with $\alpha\in \R$ has been studied by Isaev and Novikov \cite{in1}, \cite{in2}. More precisely, the authors give stability estimates for determination of potential \cite{in1} and reconstruction of a potential \cite{in2} from the Robin-to-Robin map.

General Robin-to-Robin map for bounded Lipschitz domains $\Omega \subset \R^d$, $d\geq 2$ is considered by Gesztesy and Mitrea \cite{gm1}, \cite{gm2}, where the direct Robin boundary problem for the Schr\"odinger operator $\Delta + q$ for not necessarily real-valued potentials $q$ satisfying $q\in L^\infty(\Omega)$ is solved. More concretely, the authors study the following generalized Robin boundary value problem
\begin{equation}\label{bounded}
\left\{ \begin{array}{ll}
(\Delta + q(x))u= 0 & \text{in $\Omega$}, \quad u\in H^1(\Omega)\\
\\
\frac{\partial u}{\partial \nu} + \theta u = f & \text{on $\partial \Omega$}
\end{array} \right.
\end{equation}
for every $f\in H^{-1/2}(\partial\Omega)$, where $\nu$ denotes the outward pointing normal unit vector to $\partial\Omega$. Here $\theta$ corresponds to a more general boundary operator which in particular can be the operator of multiplication by $\theta\in\R$. Note that in the bounded domain case, one can work with the standard Sobolev space
\begin{equation}
H^s(\Omega)= \{ u\in \mathcal{D}'(\Omega) : u= U|_{\Omega} \quad \text{for some} \quad U\in H^{s}(\R^d)\}, 
\end{equation}
where $\mathcal{D}'(\Omega)$ denotes the usual set of distributions on $\Omega \subset \R^d$ and $H^s(\R^d)$ is as in (\ref{sobolev}). By the unique solvability of the problem (\ref{bounded}), one can introduce the Dirichlet-to-Neumann map $\Lambda_{q,\Omega}$, the Neumann-to-Dirichlet map $N_{q,\Omega}$ and the Robin-to-Robin map $R_{\theta_1,\theta_2,\Omega}^q$ associated with $\Delta + q$ on $\Omega$ as follows,

\begin{equation}\label{dnomega}
\Lambda_{q,\Omega}: \left\{ \begin{array}{lcr}
H^{1/2}(\partial\Omega) & \longrightarrow & H^{-1/2}(\Omega) \\
u|_{\partial\Omega} & \longmapsto &  \frac{\partial u}{\partial \nu}\big{|}_{\partial\Omega},
\end{array} \right.
\end{equation}

\begin{equation}\label{ndomega}
N_{q,\Omega}: \left\{ \begin{array}{lcr}
H^{-1/2}(\Omega) & \longrightarrow & H^{1/2}(\Omega) \\
\frac{\partial u}{\partial \nu}\big{|}_{\partial\Omega} & \longmapsto & u|_{\partial\Omega},
\end{array} \right.
\end{equation}

\begin{equation}\label{rromega}
R_{\theta_1, \theta_2, \Omega}^q: \left\{ \begin{array}{lcr}
H^{-1/2}(\Omega) & \longrightarrow & H^{-1/2}(\Omega) \\
\frac{\partial u}{\partial \nu} + \theta_{1}u\big{|}_{\partial\Omega} & \longmapsto & \frac{\partial u}{\partial \nu} + \theta_{2}u\big{|}_{\partial\Omega},
\end{array} \right.
\end{equation}
where $\theta_{1}\neq \theta_{2}$. 

The question of whether the knowledge of the Robin-to-Robin map $R_{\theta_1, \theta_2, \Omega}^q$ determines $q(x)$ in an open bounded domain $\Omega \subset \R^d$, $d\geq 2$ is also of our interest.

The main result of this manuscript can be formulated as follows.

\begin{thm}\label{main}
Let $d=2, 3$. We assume that $\theta_i \geq 0$ or $\Im \theta_i >0$, $i=1,2$ such that $\theta_1 \neq \theta_2$, $\Im q_i(x) \geq 0$ in $\Rd$ and $\Im q_i(x) > 0$ in an open and bounded domain $\Omega \subset \Rd$. Suppose the set $B$ containing the supports of $q_{1}-k^{2}$ and $q_{2}-k^{2}$ is strictly contained in the upper half-space $\Rd$ and $\Omega \subset B$. Furthermore, let $q_1, q_2 \in L^\infty (B)$. If $R_{\theta_{1}, \theta_{2}}^{q_{1}} = R_{\theta_{1}, \theta_{2}}^{q_{2}}$, then $q_{1} = q_{2}$.
\end{thm}

The theorem is also true for bounded domains. In fact, we can also show the following result. 

\begin{thm}\label{domain}
Let $d\geq 2$ and $\theta_i \in \R$, $i=1,2$ with $\theta_1 \neq \theta_2$. Let $\Omega \subset \R^d$ be a open bounded domain with smooth boundary. Let $q_i\in L^{\infty}(\Omega)$, $i=1,2$. If $R_{\theta_1, \theta_2, \Omega}^{q_1} = R_{\theta_1,\theta_2,\Omega}^{q_2}$, then $q_1 = q_2$.
\end{thm}

The main idea of the proof is to reduce the Robin-to-Robin problem to a more studied Dirichlet-to-Neumann problem. We first prove that if $R_{ \theta_1, \theta_2}^{q_1}f = R_{\theta_1, \theta_2}^{q_2}f$ for a given data $f \in H^{-1/2, -\delta}(\R^{d-1})$, then $\Lambda_{q_1}g = \Lambda_{q_2}g$ for every $g\in H^{1/2,-\delta}(\R^{d-1})$. The same reasoning applies to the bounded domain case. Indeed, the same argument works for showing that the fact that $R_{\theta_1, \theta_2, \Omega}^{q_1} = R_{\theta_1, \theta_2, \Omega}^{q_2}$ implies that $\Lambda_{q_1,\Omega} = \Lambda_{q_2, \Omega}$. Thus our problem reduces to determine $q(x)$ from the knowledge of the Dirichlet-to-Neumann maps $\Lambda_q$ and $\Lambda_{q,\Omega}$. 

When $d=3$, it is known \cite{clu} that knowledge of $\Lambda_{q}$ uniquely determines $q(x)$ in the half-space geometry. As far as we know, the problem in the half-space is not studied in the two dimensional case. However, same method as in \cite{clu} allows us to extend the result to the half-plane. See Section \ref{clu2d} below. In the bounded domain case, it is known that knowledge of $\Lambda_{q, \Omega}$ uniquely determines the potential $q(x)$. For $d\geq 3$, see the seminal work of Sylvester and Uhlmann \cite{su}. The two dimensional case was open until the paper of Bukhgeim \cite{Bu} in 2008 for piecewise $W^{1,p}$ potentials with $p>2$ and later improved to $W^{\varepsilon, p}$ potentials with $\varepsilon >0$ by Bl\aa sten \cite{Bl1}. The final result is given by Imanuvilov and Yamamoto \cite{iy} for $q\in L^\infty$ and we use this in the sequel.

The paper is organized as follows. In Section \ref{1.1} we study the direct impedance boundary problem in half-space $\Rd$. Existence and uniqueness of solution of the problem (\ref{1}) for $d=2,3$ is proved. Section \ref{prooftheorem} is devoted to the proof of Theorem \ref{main}. We first give a detailed proof of how to reduce the Robin-to-Robin problem to the Dirichlet-to-Neumann one in $\Rd$ for a general dimension $d$. Same argument works for bounded domain case. See Remark \ref{remarkbounded} below. We then prove that $\Lambda_q$ uniquely determines $q$ in the half-plane. Finally, in Section \ref{appendix} (Appendix) we include a small review of the papers \cite{DMN2}, \cite{DMN1}-\cite{DMN}, \cite{cw} and \cite{cw1}-\cite{cw2} pointing out the main properties of the Green's function associated to the Helmholtz equation in $\R^3_{+}$ and $\R^2_{+}$, respectively.

\section*{Acknowledgements} 
The authors would like to express their gratitude to Katya Krupchyk for helpful and useful comments on the paper. 

\section{The direct problem}\label{1.1}

In this section we study the unique solvability of the Robin boundary problem (\ref{1}) for $d=2,3$ and for the impedance parameter $\theta$ such that $\theta >0$ (non-absorbing boundary), $\theta = 0$ (rigid boundary) and $\Im \theta >0$ (energy-absorbing boundary). 

The proof relies very much on the unperturbed Robin Green's functions given in Appendix (Section \ref{appendix}). Since there are some differences in their asymptotics and the corresponding radiation condition depending on the boundary that we consider, we split the section into two parts.

We first study the non-absorbing case $\theta >0$ in detail. Afterwards, since the rigid $\theta=0$ and energy absorbing case $\Im \theta > 0$ follows by the same method, we only focus on the differences respect to the first case. Finally, we give the existence and uniqueness result for both cases.

\subsection{Non-absorbing boundary ($\theta >0$)}\label{uni3}

Let $d=2,3$. Let us consider the Robin boundary problem
\begin{equation}\label{12}
\left\{ \begin{array}{ll}
(\Delta + q(x))u= 0 & \text{in $\Rd$}\\
\\
\frac{\partial u}{\partial x_{d}} + \theta u = f & \text{over $\{x_{d} = 0\}$},
\end{array} \right.
\end{equation}
where $u$ also satisfies the radiation condition given by
\begin{equation}\label{RC3}
\left\{ \begin{array}{ll}
\left| \frac{\partial u}{\partial r} - iku\right| < \frac{C}{r^{\left(2\alpha + \frac{1}{2} \right)}} & \text{in $\R^3_{+}(\alpha_{+}) = \{x_{3} > r^{\alpha}\}$}\\
\\
\left|\frac{\partial u}{\partial r} - i \sqrt{k^{2} + \theta^{2}}u \right| < \frac{C}{r^{\left(\frac{3}{2} -\alpha \right)}} & \text{in $\R^3_{+}(\alpha_{-}) = \{0 \leq  x_{3} < r^{\alpha} \}$},
\end{array} \right.
\end{equation}
when $r \to + \infty$, for any $\alpha \in \left(\frac{1}{4}, \frac{1}{2} \right)$ and
\begin{equation}\label{RC2}
\left\{ \begin{array}{ll}
\left| \frac{\partial u}{\partial r} - iku\right| < \frac{C}{r^{1+\alpha}} & \text{in $\R^2_{+}(\alpha_{+}) = \{x_{2} > r^{\alpha}\}$}\\
\\
\left|\frac{\partial u}{\partial r} - i \sqrt{k^{2} + \theta^{2}}u \right| < \frac{C}{r^{1+\alpha}} & \text{in $\R^2_{+}(\alpha_{-}) = \{0 \leq  x_{2} < r^{\alpha} \}$},
\end{array} \right.
\end{equation}
when $r \to + \infty$, for any $\alpha \in \left(0,\frac{1}{2} \right)$.

Our first goal is to construct the unique solution of the problem (\ref{12}) using the Green's function given by Dur\'an, Muga and  N\'ed\'elec \cite{DMN2}, \cite{DMN1}-\cite{DMN}. See Appendix \ref{sectiongreen} below.

Let $G_{\theta}(x,y)$ denote the unperturbed Robin Green's function satisfying the boundary problem
\begin{equation}\label{G1}
\left\{ \begin{array}{ll}
(\Delta + k^{2})G_{\theta}(x, y)= -\delta(x-y) & \text{in $\Rd$}\\
\\
\frac{\partial G_{\theta}(x,y)}{\partial x_{d}} + \theta G_{\theta}(x,y) = 0 & \text{over $\{x_{d} = 0\}$}\\
\\
G_{\theta} \, \text{satisfies the radiation condition (\ref{RC3}) or (\ref{RC2})},
\end{array} \right.
\end{equation}
for $x, y\in \Rd$.

We construct the solution to the problem (\ref{12}) as the solution to the following integral equation
\begin{equation}\label{+}
u(x) =  \frac{1}{2 C_\pi C'_\pi} \int_{y_{d}=0}  G_{\theta}(x,y)f(y)\, d\sigma(y) + \int_{y_{d} > 0} G_{\theta}(x,y) \left( q(y) - k^{2} \right)u(y)\, dy,
\end{equation}
where $C_\pi = \frac{1}{\sqrt{8\pi}}, \, C'_\pi= \frac{1}{4\pi}$ if $d=2$ and $C_\pi = \frac{1}{4\pi},\,  C'_\pi = \frac{1}{2\pi}$ when $d=3$.

We first need to show that $u$ has the desired properties. It clearly satisfies the corresponding radiation condition. To see that the Robin boundary condition $\frac{\partial u}{\partial x_d} + \theta u = f$ holds at $x_d = 0$, observe that the entire contribution to $u$ comes from the first term on the right-hand side of (\ref{+}) because of the Robin boundary condition that satisfies the Green's function if $x_d = 0$ and $y_d >0$. Thus we need to check that when $y_d = 0$, we obtain
\begin{equation}\label{poisson0}
\frac{\partial G_\theta(x,y)}{\partial x_d} + \theta G_\theta(x,y) = 2 C_\pi C'_\pi\delta(x'-y') \quad \quad \text{over} \quad \{x_d = 0\},
\end{equation}
where $x=(x', x_d)$, $y=(y', y_d)$ and $C_\pi, \, C'_\pi$ are as above. For this purpose, from (\ref{fouriergreen}) and (\ref{spatialrobin}), observe that on the surface $y_d=0$ yields
\begin{equation}\notag
G_\theta(x,y) = 2 C_\pi C'_\pi\int_{\R^{d-1}}\frac{1}{\theta -\sqrt{\xi^2 - k^2}}e^{-\sqrt{\xi^2 - k^2}x_d}e^{-i\xi\cdot (x'-y')} \, d\xi
\end{equation}
and
\begin{equation}\notag
 \frac{\partial G_\theta(x,y)}{\partial x_d} = -2C_\pi C'_\pi\int_{\R^{d-1}}\frac{\sqrt{\xi^2 - k^2}}{\theta -\sqrt{\xi^2 - k^2}}  e^{-\sqrt{\xi^2 - k^2}x_d}e^{-i\xi\cdot (x'-y')} \, d\xi.
\end{equation}
Hence, evaluating at $x_d = 0$, (\ref{poisson0}) follows.

\begin{remark}
We point out that when dealing with Robin Green's function $G_\theta$, by (\ref{fouriergreen}) and (\ref{spatialrobin}) the solution $u$ can be also defined as the solution of the integral equation 
\begin{equation}\label{+'}
u(x) = - \frac{1}{2\theta C_\pi C'_\pi} \int_{y_{d}=0} \frac{ \partial G_{\theta}(x,y)}{\partial y_d}f(y)\, d\sigma(y) + \int_{y_{d} > 0} G_{\theta}(x,y) \left( q(y) - k^{2} \right)u(y)\, dy,
\end{equation}
being $C_\pi$ and $C'_\pi$ as above and $\theta \neq 0$. It is easy to check that when $y_d = 0$, yields
\begin{equation}\label{poisson}
\frac{\partial}{\partial x_d}\frac{\partial G_\theta(x,y)}{\partial y_d} + \theta\frac{\partial G_\theta(x,y)}{\partial y_d} = -2\theta C_\pi C'_\pi\delta(x'-y') \quad \quad \text{over} \quad \{x_d = 0\}.
\end{equation}
This follows by showing that on the surface $y_d=0$, we have
\begin{equation}\notag
\frac{\partial G_\theta(x,y)}{\partial y_d} = -C_\pi C'_\pi\int_{\R^{d-1}}\frac{2\theta}{\theta -\sqrt{\xi^2 - k^2}}e^{-\sqrt{\xi^2 - k^2}x_d}e^{-i\xi\cdot (x'-y')} \, d\xi
\end{equation}
and
\begin{equation}\notag
\frac{\partial}{\partial x_d} \frac{\partial G_\theta(x,y)}{\partial y_d} = C_\pi C'_\pi\int_{\R^{d-1}}\frac{2\theta\sqrt{\xi^2 - k^2}}{\theta -\sqrt{\xi^2 - k^2}}  e^{-\sqrt{\xi^2 - k^2}x_d}e^{-i\xi\cdot (x'-y')} \, d\xi,
\end{equation}
which evaluating at $x_d = 0$ implies (\ref{poisson}).
\end{remark}

In order to prove existence and uniqueness of solution of (\ref{12}), we first need to show that equation (\ref{+}) is uniquely solvable in a suitable function space. Indeed, we prove that (\ref{+}) has a unique solution in $H^{1,-\delta}(\Rd)$ with $\delta>\frac{1}{2}$ by applying the Fredholm alternative. For simplicity of notation, we write $V(y) = q(y) - k^{2}$ and let $G_{\theta}V$ stand for the operator given by 
$$
G_{\theta}V u(x)= \int_{y_{d} > 0} G_{\theta}(x,y) V(y) u(y) \, dy.
$$

\begin{pro}\label{pro1}
Let $d=2,3$. If $V$ is a bounded function of compact support, the operator $G_{\theta}V$ is compact in $H^{1,-\delta}(\Rd)$ for any $\delta>\frac{1}{2}$.
\end{pro}

\begin{proof}

From the assumption that $V$ has compact support in the lower half-space, we write $G_{\theta}V$ as $G_{\theta}\chi V$, where $\chi$ is the function that is one on the support of $V$ and zero everywhere else. Let $r_1, r_2 >0$ such that $supp \, V \subset \Omega \subset \{ r_1 \leq |y| \leq r_2\}$.

As in Agmon \cite{A}, we first note that multiplication by $V$ is a compact operator mapping $H^{1,-\delta}(\Rd)$ into $L^{2,\delta}(\Rd)$ and then show that the operator $G_{\theta}\chi$  is a bounded operator mapping $L^{2,\delta}(\Rd)$ into $H^{1,-\delta}(\Rd)$. Hence the product operator $G_{\theta}\chi V = G_{\theta}V$ is a compact operator on $H^{1,-\delta}(\Rd)$.

Multiplication by $V$ is a compact operator from $H^{1,-\delta}$ to $L^{2,\delta}$ by the Sobolev embedding theorem.

To show that $G_{\theta}\chi$ is bounded from $L^{2,\delta}$ into $H^{1,-\delta}$, we use a functional analysis approach. Let
\begin{equation}
k(x,y) = \frac{G_\theta(x,y)\chi(y)}{(1+|x|^2)^{\delta/2}(1+|y|^2)^{\delta/2}}
\end{equation}
denote a function $k:\Rd \times \Rd \to \mathbb{C}$ and we define the associated integral operator
$$
Kf(x)=\int_{\Rd} k(x,y)f(y)dy.
$$
We need to prove that $K: L^2 \to H^1$ is bounded.

We first show that for $\delta > 1/2$ 
\begin{equation}\label{goal21}
\int_{\Rd}\int_{\Rd} |k(x,y)|^2 \, dy \, dx < +\infty.
\end{equation}
This implies that $K:L^2 \to L^2$ is a Hilbert Schmidt operator and hence, compact and bounded. Secondly, we prove that the integral kernel 
$$
h_i(x,y)= \frac{\partial}{\partial x_i}k(x,y)\quad \quad \quad i=1,\ldots, d
$$
defines, as well, a bounded integral operator $H_i$ from $L^2$ to $L^2$. This together with \eqref{goal21} shows that $K: L^2 \to H^1$ is bounded.

To show \eqref{goal21} note that by (\ref{grb20}), (\ref{grb21}), (\ref{grb}), the Robin Green's function holds 
\begin{equation}
|G_\theta(x,y)|\leq \frac{C}{|x-y|^{\frac{d-1}{2}}}.
\end{equation}
Let $\varphi \in C_0^\infty(\Rd)$ be such that $\varphi(x)=1$ for $|x|\leq r_2 + 1$ and $\varphi(x)=0$ for $|x|\geq 2r_2 + 1$. Thus for any $\delta > 0$ and $d=2,3$, we get
\begin{align}\notag
\int_{\Rd} \int_{\Rd} |\varphi(x)k(x,y)|^2 \, dx\, dy & \leq C \int_{\Omega} \frac{dy}{(1+|y|^2)^\delta} \int_{\tilde{\Omega}\cap \{|x-y|\leq 1\}} \frac{dx}{|x-y|^{d-1}}\notag\\ 
& + \int_{\Omega} \frac{dy}{(1+|y|^2)^\delta} \int_{\tilde{\Omega}\cap \{|x-y|\geq 1\}}  \frac{dx}{(1+|x|^2)^\delta} < \infty,\notag
\end{align}
where $\tilde{\Omega}$ is a bounded domain containing the support of $\varphi$.
In addition, since $1-\varphi(x) = 0$ when $|x|\leq r_2 +1$ and in the support of $V$, using that $|x-y|\geq |x|-r_2$, we have
\begin{align}\notag
\int_{\Rd} \int_{\Rd} |(1-\varphi(x))k(x,y)|^2 \, dx\, dy & \leq C \int_\Omega \frac{dy}{(1+|y|^2)^\delta}\int_{|x|\geq r_2 +1} \frac{dx}{|x|^{2\delta}|x-y|^{d-1}}\\
& \leq C\int_{|x|\geq r_2 +1} \frac{dx}{||x|-r_2|^{2\delta +d-1}} < \infty,\notag 
\end{align}
for $2\delta - 1 >0$ i.e. $\delta > 1/2$.

To prove the claim for $H_i$ note that
\begin{equation}\notag
|h_i(x,y)| \leq \frac{C}{(1+|y|^2)^{\delta/2}}\left(\frac{1}{(1+|x|^2)^{\frac{\delta}{2}}|x-y|^{\frac{d+1}{2}}} + \frac{1}{(1+|x|^2)^{\frac{\delta+1}{2}}|x-y|^{\frac{d-1}{2}}} \right).
\end{equation}
The second term of the above inequality is square integrable for all $\delta >0$ and defines thus a Hilbert Schmidt operator, as above. Similarly, $(1-\varphi(x))h_i(x,y)$ is square integrable for any $\delta >0$. The proof is completed by applying the Schur test to the first term of $\varphi(x)h_i(x,y)$, i.e. checking that 
\begin{equation}\notag
\sup_{|x|\leq 2r_2+1} \frac{1}{(1+|x|^2)^{\delta/2}} \int_{|y|\leq r_2} \frac{dy}{(1+|y|^2)^{\delta/2}|x-y|^{\frac{d+1}{2}}} < \infty
\end{equation}
and
\begin{equation}\notag
\sup_{|y|\leq r_2} \frac{1}{(1+|y|^2)^{\delta/2}} \int_{|x|\leq 2r_2+1} \frac{dx}{(1+|x|^2)^{\delta/2}|x-y|^{\frac{d+1}{2}}} < \infty.
\end{equation}
These inequalities are true for all $\delta >0$ and $d=2,3$ since
\begin{align}\notag
\sup_{|x|\leq R} \int_{|y|\leq R'} \frac{dy}{|x-y|^{\frac{d+1}{2}}} \leq \sup_{|x|\leq R} \int_{|z|\leq R+R'} \frac{dz}{|z|^{\frac{d+1}{2}}} < \infty
\end{align}
for any $R, R' >0$ and $d>1$.

\end{proof}

Now we are ready to prove the uniqueness result for the integral equation (\ref{+}).

\begin{pro}\label{++}
Let $d=2,3$, $\Im q(x) \geq 0$ in $\Rd$ and $\Im q(x) >0$ in an open and bounded domain $\Omega \subset \Rd$. Suppose that $V$ is a bounded function with compact support. Then equation (\ref{+}) has a unique solution $u$ in $H^{1,-\delta}(\Rd)$.
\end{pro}

\begin{proof}
By Proposition \ref{pro1}, the Fredholm alternative guarantees that (\ref{+}) has a unique solution provided that the corresponding homogeneous equation has only the zero solution. In what follows, $B$ stands for a compact set in $\Rd$ that contains the support of $V$.

Observe that a solution of the homogeneous equation
\begin{equation}
u(x) = \int_{y_{d}> 0} G_{\theta}(x,y) V(y) u(y)\, dy
\end{equation}
is also a solution of the equation
\begin{equation}\label{2.4}
(\Delta + q(x)) u(x) = 0 \quad \text{in} \quad \Rd, 
\end{equation}  
with boundary condition
\begin{equation}\label{robin}
\frac{\partial u}{ \partial x_{d}} + \theta u = 0 \quad \text{over} \quad \{x_{d}=0\}.
\end{equation}

To show that such a $u$ must be identically zero, we use an energy identity. We first introduce some notation and see some properties related to the radiation condition. Let $\sken$ denote the surface of the half-sphere of radius $R$ contained in the half-space $\Rd$. Let $\sbat$ be the part of $\sken$ contained in the domain $\Rgehi$. Let $\sbi$ be the complementary part, i.e. the part of $\sken$ contained in the domain $\Rken$. Note that the radiation conditions (\ref{RC3}) and (\ref{RC2}) imply that
\begin{equation}\label{radisphere}
\left\{ \begin{array}{ll}
\lim_{R \to +\infty} \int_{\sbat} \left\vert \frac{\partial u}{\partial r} - iku \right\vert^{2}\, dS = 0 \\
\\
\lim_{R \to +\infty} \int_{\sbi} \left\vert \frac{\partial u}{\partial r} - i\sqrt{\theta^2 + k^2}u \right\vert^{2}\, dS = 0. \\
\end{array} \right.
\end{equation}

The procedure for obtaining the energy identity is to multiply (\ref{2.4}) by the complex conjugate $\bar{u}$ and integrate over the domain within $\sken$ and the plane $x_d =0$. We call this domain $B_R^{+}$ and we choose $R$ such that $B \subset \bken$. After an application of the divergence theorem, we get
\begin{equation}\label{a2}
\int_{B_R^{+}} \left( |\nabla u|^{2} - q(x) |u|^{2} \right) \, dx = \int_{\sken} \frac{\partial u }{\partial r} \bar{u}\, dS - \int_{\{x_d =0\}\cap \bken}\frac{\partial u }{\partial x_d} \bar{u}\, d\sigma(x).
\end{equation}
We now take the imaginary part of the above equality. By the Robin boundary condition (\ref{robin}), since $\theta \in \R$, it follows that the integral over $\{x_d = 0\}$ does not contribute and we obtain
\begin{align}
-\int_{\bken} \Im q(x) \vert u \vert^2 \, dx &= \Im \int_{\sken} \frac{\partial u}{\partial r}\bar{u}\, dS \label{sr}\\
& = \Im \int_{\sbat} \left[\frac{\partial u}{\partial r} - iku \right]\bar{u}\, dS + k\int_{\sbat} \vert u \vert^2 \, dS\notag\\
& + \Im \int_{\sbi} \left[\frac{\partial u}{\partial r} - i \sqrt{k^2 + \theta^2}u \right]\bar{u}\, dS + \sqrt{\theta^2 + k^2}\int_{\sbi} \vert u \vert^2 \, dS.\notag
\end{align}

Observe that on the one hand, writting $q(x) = V(x) + k^2$, since $k\in \R$ and $supp \, V \subset B$, the left-hand side of the above identity can be given by
\begin{equation}\label{v}
-\Im\int_{B} V(x)|u|^2\, dx,
\end{equation}
which is finite as $V$ is a bounded potential and $u\in L^2(B)$. In addition, by the assumption that $\Im q(x) >0$, we get that (\ref{v}) is negative. On the other hand, letting $R\to \infty$, we obtain
\begin{equation}\label{lim}
\lim_{R \to \infty} -\int_{\bken} \Im q(x)|u|^2 \, dx = -\int_{\Rd} \Im q(x) |u|^2\, dx.
\end{equation}

To deal with the right-hand side of (\ref{sr}), we first note that its boundedness implies that in particular $\int_{S_R^i} |u|^2 \, dS < \infty$, $i=1,2$. Thus by the radiation condition (\ref{radisphere}) we get
\begin{equation}\label{li11}
\lim_{R \to \infty} \left[\Im \int_{\sbat} \left(\frac{\partial u}{\partial r} - iku \right)\bar{u}\, dS +\Im \int_{\sbi} \left(\frac{\partial u}{\partial r} - i \sqrt{k^2 + \theta^2}u \right)\bar{u} \right]\, dS = 0.
\end{equation}
Hence, according to the above remark related to the fact that (\ref{v}) is negative and finite, by using the positivity of $k$, $\sqrt{\theta^2 + k^2}$, it follows that
$$
0 < \lim_{R \to +\infty} \left[k \int_{\sbat} \vert u \vert^2\, dS + \sqrt{\theta^2 + k^2}\int_{\sbi} \vert u \vert^2 \, dS \right] \leq 0.
$$
This together with (\ref{li11}) implies that the right-hand side of (\ref{sr}) tends to zero as $R \to \infty$. As a consequence, combining this with (\ref{lim}) gives
\begin{equation}
\int_{\Rd} \Im q(x) |u|^2 \, dx= 0.
\end{equation}

Finally, since $\Im q(x) > 0$ in $\Omega \subset \Rd$, it may be concluded that $u$ must be zero there, and the unique continuation gives $u=0$ in $\Rd$, which is our claim.

\end{proof}

\subsection{Rigid and absorbing boundary ($\theta = 0$ and $\Im \theta >0$)}\label{uni2}

Let us consider the problem (\ref{12}) in half-space $\Rd$, with the radiation conditions (\ref{RC3}) and (\ref{RC2}) replaced by the usual Sommerfeld's one,
\begin{equation}\label{S}
\frac{\partial u}{\partial r} -i ku = o(r^{-\frac{(d-1)}{2}}), 
\end{equation}
uniformly in $\alpha$ as $r=|x| \to \infty$ with $0 < \alpha < \pi$, where here $(r, \alpha)$ denotes the plane polar coordinates of $x\in \Rd$. 


The same approach as above applies to these cases. By using (\ref{grbcw0}) and (\ref{grbcw1}), the proof of Propositions \ref{pro1} and \ref{++} runs as before, the only difference being in the uniqueness result when $\Im \theta >0$. 

In order to prove that the integral equation (\ref{+}) for $\Im \theta >0$ has a unique solution, one needs to work a bit more. After multiplying equation $\Delta u + q(x)u = 0$ by $\bar{u}$ and integrating over the domain $B_{R}^+$, since in this case $\theta \notin \R$, when we take the imaginary part of the identity (\ref{a2}), we obtain
\begin{align}\label{tita2}
-\Im \int_{B_{R}^+}& q(x)|u|^2 \, dx = \Im \int_{\sken} \frac{\partial u}{\partial r}\bar{u} \, dS + \Im\theta \int_{\{x_d =0\}\cap \bken} |u|^2 \, d\sigma(x)\\
& = \Im \left[ \int_{\sken} \left(\frac{\partial u}{\partial r} -iku \right)\bar{u}\, dS + ik\int_{\sken} |u|^2\, dS + \theta \int_{\{x_d =0\}\cap \bken} |u|^2 \, d\sigma(x)\right].\notag
\end{align}

We deal with the left hand-side of the above identity as in the $\theta >0$ case. From the fact that $u \in L^2_{loc}(\R^2_{+})$ and $\Im q > 0$ in $B \subset B_R^+$, we have that the right hand-side of (\ref{tita2}) is also bounded. In particular, $\int_{\sken} |u|^2 \, dS < \infty$. Thus by the radiation condition (\ref{S}), it follows that
\begin{equation}
\lim_{R \to \infty} \Im \int_{\sken} \left(\frac{\partial u}{\partial r} -iku \right)\bar{u}\, dS = 0.
\end{equation}
In addition, since for $R$ large enough, the left-hand side of (\ref{tita2}) is negative and the right-hand side is positive, we have
\begin{equation}
 \lim_{R \to \infty} \left[ k\int_{\sken} |u|^2 \, dS + \Im\theta \int_{\{x_d =0\}\cap\bken} |u|^2 \, d\sigma(x) \right] = 0.
\end{equation}

The rest of the proof runs as before.

\begin{remark}
Note that this argument breaks down in the case that $\Im k < 0$ or/and $\Im \theta < 0$. This is the reason that we are not able to show the case when $\Re\theta < 0$ which is studied in \cite{DHN}. Our argument would apply for $k\in \mathbb{C}$ with $\Im k > 0$ and $\theta \in \mathbb{C}$ with $\Im \theta >0$. The rest of the cases need a different reasoning.
\end{remark}

\subsection{Existence and uniqueness}\label{exun}

We can now formulate our main result of this section.

\begin{thm}
Let $d=2,3$. We assume that $\theta\geq 0$ or $\Im \theta >0$, $\Im q(x)\geq 0$ in $\Rd$ and $\Im q(x) >0$ in a bounded domain $\Omega \subset \Rd$. Let $f$ be a function in $H^{-1/2, -\delta}(\R^{d-1})$, $\delta >\frac{1}{2}$. Then the problem (\ref{1}) has a unique solution $u \in H^{1, -\delta}(\Rd)$ .
\end{thm}

\begin{proof}

\begin{itemize}
\item[]

\item[] \emph{Existence:} We will show that the solution $u$ of the integral equation (\ref{+}) satisfies problem (\ref{1}). 

By construction, since $G_{\theta}$ is solution of the problem (\ref{G1}), u satisfies equation (\ref{2.4}), the boundary condition
$$
\frac{\partial u}{\partial x_{d}} + \theta u = f \quad \text{over} \quad \{ x_{d} = 0\}
$$
and the radiation conditions (\ref{RC3}), (\ref{RC2}) or (\ref{S}). 

\item[] \emph{Uniqueness:} The uniqueness is guaranteed by showing that $u\equiv 0$ if $f=0$. This is done in the proof of Proposition \ref{++}.
\end{itemize}
\end{proof}

\section{Proof of the main result}\label{prooftheorem}

This section deals with the proof of Theorem \ref{main}. 

It will be divided into two steps. We first prove that knowledge of the Robin-to-Robin map implies knowledge of the Dirichlet-to-Neumann map. This argument works for both two and three dimensions, and any $\theta \in \mathbb{C}$. We next concern with determining the potential $q(x)$ in the half-space from the Dirichlet-to-Neumann map. In the three dimensional case, this is done by Cheney, Lassas and Uhlmann \cite{clu}. Thus we leave it to the reader. We will focus on $d=2$, giving a detailed proof for this case.

\subsection{From Robin-to-Robin to Dirichlet-to-Neumann}\label{fromrrtodn}

Let $R_{\theta_1, \theta_2}^{q_i}$, $\Lambda_{q_i}$, $N_{q_i}$ denote the Robin-to-Robin map, the Dirichlet-to-Neumann map and the Neumann-to-Dirichlet map associated with $\Delta + q_i$ on $\Rd$, defined by (\ref{dnmap}), (\ref{ndmap}) and (\ref{rrmap}), respectively.

Our main goal here is to show the following result. Unless otherwise stated, we work under the assumptions of Theorem \ref{main}.

\begin{pro}\label{rn3}
Let $f\in H^{-1/2, -\delta}(\R^{d-1})$. If $R_{\theta_1, \theta_2}^{q_1}f = R_{\theta_1, \theta_2}^{q_2}f$, then $\Lambda_{q_1}(u_1|_{x_d=0}) = \Lambda_{q_2}(u_2|_{x_d=0})$, where $u_i\in H^{-1,-\delta}(\Rd)$, $i=1,2$ is the unique solution of the problem
\begin{equation}\label{rob1}
\left\{ \begin{array}{ll}
(\Delta + q_i(x))u_i= 0 & \text{in $\Rd$}\\
\frac{\partial u_i}{\partial x_{d}} + \theta_1 u_i = f & \text{over $\{x_{d} = 0\}$}\\
u_i \quad \text{satisfies the corresponding radiation condition}.
\end{array} \right.
\end{equation}
As a consequence, $\Lambda_{q_1}g = \Lambda_{q_2}g$ for any $g\in H^{1/2,-\delta}(\R^{d-1})$.
\end{pro}







To this end, we first give the following result that relates the mappings $R_{\theta_1, \theta_2}^q$ and $\Lambda_q$ and is fundamental for our approach.

\begin{lem}\label{lema2}
Let $\theta, \theta' \in \mathbb{C}$ with $\theta \neq \theta'$ and $S = R_{\theta, \theta'}^{q} - I:H^{-1/2,-\delta}(\R^{d-1})\to H^{-1/2,-\delta}(\R^{d-1})$. Then
\begin{equation}\label{sq}
S(\Lambda_{q} + \theta)= (\theta - \theta')I .
\end{equation}
Moreover,  
\begin{itemize}
\item[(i)] $\Lambda_{q} + \theta$ is injective.
\item[(ii)]$\Lambda_{q} + \theta$ has a dense range.
\item[(iii)] $(\Lambda_{q} +\theta)S=(\theta - \theta')I$.
\item[(iv)] $S$ is one to one with $S^{-1}= (\theta - \theta')(\Lambda_q + \theta)$. Hence, $\Lambda_q + \theta$ is invertible.
\end{itemize}
\end{lem}

\begin{proof}
Let $f\in H^{-1/2,-\delta}(\R^{d-1})$. We denote $u_f = u|_{x_d =0}$ and $\frac{\partial u_f}{\partial x_d}= \frac{\partial u}{\partial x_d}\Big|_{x_d=0}$, where $u\in H^{-1,-\delta}(\Rd)$ is the unique solution of the Robin problem
$$
(\Delta + q)u = 0 \quad \text{in} \quad \Rd \, , \quad \quad \frac{\partial u}{\partial x_d} + \theta u = f \quad \text{over} \quad \{x_d = 0\}.
$$

An easy computation shows that 
\begin{align}
S(\Lambda_{q} + \theta) u_f & = R_{\theta, \theta'}^{q} (\Lambda_{q}u_f + \theta u_f) - \left(\frac{\partial u_f}{\partial x_d} + \theta u_f \right)\notag\\
& = \left(\frac{\partial u_f}{\partial x_d} + \theta' u_f \right) - \left(\frac{\partial u_f}{\partial x_d} + \theta u_f\right)\notag\\
& = (\theta' - \theta)u_f,\notag
\end{align}
which gives (\ref{sq}).

To deal with the injectivity of the operator $\Lambda_{q} + \theta$, let $g\in H^{1/2,-\delta}(\R^{d-1})$ and consider the Dirichlet problem
\begin{equation}\label{dir01}
\left\{ \begin{array}{ll}
(\Delta + q(x))u= 0 & \text{in $\Rd$}\\
u_{f} = g \\
u \quad \text{satisfies the corresponding radiation condition},
\end{array} \right.
\end{equation}
so that $\Lambda_q g = \Lambda_q u_f$. Now observe that if $(\Lambda_{q} +\theta)u_f = 0,$ then $S(\Lambda_{q} + \theta)u_f = 0$ and by (\ref{sq}) we have $(\theta' - \theta) u_f = 0$. Hence, since $\theta \neq \theta'$, it follows that $u_f=0$, which is our claim.

We next prove that if $v\in H^{1/2, -\delta}(\R^{d-1})$ is such that $\langle (\Lambda_{q} +\theta)u, v \rangle=0$ $\forall u \in H^{1/2, -\delta}(\R^{d-1})$, then $v=0$. Observe that for any $u\in H^{1/2, -\delta}(\R^{d-1})$, yields
\begin{align}\notag
\langle (\Lambda_{q} +\theta)u, v \rangle=0 & \Leftrightarrow \langle u, (\Lambda_{q} +\overline{\theta})v \rangle=0.
\end{align}
Thus $(\Lambda_{q} +\overline{\theta}) v = 0$ and by (i) we obtain $v=0$ and (ii) follows.

In order to get (iii), using $f=\frac{\partial u_f}{\partial x_d} + \theta u_f$, we check that
\begin{align}
(\Lambda_{q} + \theta) Sf & = (\Lambda_q + \theta)(R_{\theta,\theta'}-I)\left(\frac{\partial u_f}{\partial x_d}+\theta u_f \right)\notag\\
& = (\Lambda_q + \theta)(\theta'-\theta)u_f\notag\\
& = (\theta' -\theta)f\notag
\end{align}
Then by property (ii) we extend the operator to $H^{-1/2,-\delta}(\R^{d-1})$.

Finally, by (\ref{sq}) and (iii) it follows that S is surjective and injective respectively, thus $S$ is invertible with $S^{-1} = (\Lambda_q + \theta)$ and the proof is completed.
\end{proof}

Now we are ready to prove the main result of this section.

\begin{proof}{\bf{Proof of Proposition \ref{rn3}}}

Our proof starts with the observation that
\begin{align}
& R_{\theta_{1}, \theta_{2}}^{q_{1}} \left( \Lambda_{q_{1}} + \theta_{1} \right) = \Lambda_{q_{1}} + \theta_{2},\label{r1} \\
& R_{\theta_{1}, \theta_{2}}^{q_{2}} \left( \Lambda_{q_{2}} + \theta_{1} \right) = \Lambda_{q_{2}} + \theta_{2}\label{r2}.
\end{align}

By the assumption $R_{\theta_{1}, \theta_{2}}^{q_{1}} = R_{\theta_{1}, \theta_{2}}^{q_{2}}$ and the fact that $R_{\theta_{1}, \theta_{2}}^{q_{i}}$ is a linear operator, subtracting (\ref{r1}) and (\ref{r2}) we obtain
\begin{equation}\label{r3}
R_{\theta_{1}, \theta_{2}} \left(\Lambda_{q_{1}} - \Lambda_{q_{2}}\right) =  \Lambda_{q_{1}} - \Lambda_{q_{2}}.
\end{equation}
Here and subsequently, $R_{\theta_{1}, \theta_{2}}$ stands for $R_{\theta_{1}, \theta_{2}}^{q_{i}}$, $i=1,2$.

Let us denote $T= R_{\theta_{1}, \theta_{2}} - I$ and $u_{1_{f}}=u_1|_{x_d=0}$, $u_{2_{f}}=u_{2}|_{x_d=0}$. Thus from (\ref{r3}) we have
\begin{equation}
T(\Lambda_{q_{1}}u_{1_{f}} - \Lambda_{q_{2}}u_{2_{f}}) = 0
\end{equation}
and by Lemma \ref{lema2} using that $T$ is injective, we obtain 
\begin{equation}\label{dn11}
\Lambda_{q_{1}}u_{1_{f}} - \Lambda_{q_{2}}u_{2_{f}} = 0,
\end{equation}
which proves the first part of the proposition.

Let $g\in H^{1/2,-\delta}(\R^{d-1})$. The proof is completed by showing that $\Lambda_{q_1}g = \Lambda_{q_2}g$. For that, we consider the Dirichlet problem
\begin{equation}\label{dir1}
\left\{ \begin{array}{ll}
(\Delta + q_i(x))u_i= 0 & \text{in $\Rd$}\\
u_{i_f} = g \\
u_i \quad \text{satisfies the corresponding radiation condition},
\end{array} \right.
\end{equation}
where $u_i$ is the unique solution of the Robin problem (\ref{rob1}) and $u_{i_f}= u_i|_{x_d = 0}$, $i=1,2$. Observe that this can be seen as the particular case of (\ref{rob1}) with $\theta=\infty$. Thus for each $g\in H^{1/2,-\delta}(\R^{d-1})$ it is guaranteed the existence of the unique solution of the problem (\ref{dir1}) and the associated Dirichlet-to-Neumann map $\Lambda_{q_i}$ is given by 
$$
\Lambda_{q_i} g = \Lambda_{q_i}u_{i_f}.
$$
This together with (\ref{dn11}) gives our claim and the proof is complete.

\end{proof}

\begin{remark}\label{remarkbounded}
The same reasoning applies to the sufficiently smooth bounded domain case. In the same manner, using the maps defined by (\ref{dnomega}), (\ref{ndomega}) and (\ref{rromega}), we can show that for a given $f\in H^{-1/2}(\partial\Omega)$ if $R_{\theta_1, \theta_2, \Omega}^{q_1}f = R_{\theta_1, \theta_2, \Omega}^{q_2}f$, then $\Lambda_{q_1, \Omega}g = \Lambda_{q_2, \Omega}g$, for all $g\in H^{1/2}(\partial\Omega)$ which proves Theorem \ref{domain}. 
\end{remark}

It is worth pointing out that in a bounded domain case there is an alternative approach to prove that $\Lambda_{q,\Omega} + \theta$ is one to one. In fact, it can be shown that the operator $\Lambda_{q,\Omega} + \theta$ is Fredholm with index zero. For this purpose, on the one hand it needs to be checked that $\Lambda_{q,\Omega}N_{q, \Omega} = I = N_{q, \Omega}\Lambda_{q, \Omega}$. This is done by Gesztesy and Mitrea, see Remark 3.8 \cite{gm1} or Remark 3.6 \cite{gm2}. On the other hand, the compact embedding $H^{1/2}(\partial\Omega) \hookrightarrow H^{-1/2}(\partial\Omega)$ for sufficiently smooth bounded domains $\Omega \subset \R^d$ is used.

This argument breaks down in the half-space case. In particular, the embedding $H^{1/2,-\delta}(\R^{d-1}) \longrightarrow H^{-1/2,-\delta}(\R^{d-1})$ is not compact. However, it can be proved that the Dirichlet-to-Neumann map and the Neumann-to-Dirichlet map are inverse operators in our setting.

\begin{pro}\label{key}
Let $\delta > 1/2$. For any $f\in H^{1/2, -\delta}(\R^{d-1})$, it follows that
\begin{equation}\label{dnd}
\Lambda_{q} N_{q} f = f = N_{q} \Lambda_{q} f.
\end{equation}
\end{pro}

\begin{proof}
Let $f\in H^{1/2, -\delta}(\R^{d-1})$. We first show that
\begin{equation}\label{dnda}
N_{q} \Lambda_{q} f = f.
\end{equation}

For a given data $f\in H^{1/2, -\delta}(\R^{d-1})$, we consider the following Dirichlet problem:
\begin{equation}\label{dirichlet}
\left\{ \begin{array}{ll}
(\Delta + q) u= 0 & \text{in $\Rd$}\\
\\
u = f & \text{over $\{x_{d} = 0\}$}\\
\\
\lim_{R \to +\infty} \int_{|x|=R} \left\vert \frac{\partial u}{\partial r} - iku \right\vert^{2} \, dS = 0 .
\end{array} \right.
\end{equation}

As mentioned above, this is just the particular case of our Robin problem (\ref{1}) when $\theta = \infty$. Thus we know that there exists a unique solution $u_{f} \in H^{1, -\delta}(\R^{d-1})$ of (\ref{dirichlet}) and we can define the associated Cauchy-data set  as in (\ref{cauchy}) such that $C_{q} \subset H^{1/2, -\delta}(\R^{d-1}) \times H^{-1/2, -\delta}(\R^{d-1})$, as well as the maps $\Lambda_{q}$, $N_{q}$ as in (\ref{dnmap}), (\ref{ndmap}), respectively.

Let $f_{u}=u_{f}|_{x_d =0}$ and $g_{u} = \Lambda_{q} u_{f} = \frac{\partial u_{f}}{\partial x_{d}}|_{x_{d}=0}$. Observe that by definition, 
$$
N_{q} g_{u} = f_{u}, \quad \quad \quad \Lambda_{q} f_{u} = g_{u}.
$$
Hence, 
\begin{align}
N_{q}\Lambda_{q} f = N_{q} \Lambda_{q} f_{u} = N_{q} g_{u} = f_{u} = f,
\end{align}
which gives (\ref{dnda}).

The same argument works for showing that 
\begin{equation}
\Lambda_{q} N_{q} f = f,
\end{equation}
with the only difference that in this case we consider the associated Neumann problem
\begin{equation}\label{neumann}
\left\{ \begin{array}{ll}
(\Delta + q) u= 0 & \text{in $\Rd$}\\
\\
\frac{\partial u}{\partial x_{d}} = f & \text{over $\{x_{d} = 0\}$}\\
\\
\lim_{R \to +\infty} \int_{|x|=R} \left\vert \frac{\partial u}{\partial r} - iku \right\vert^{2} \, dS = 0 .
\end{array} \right.
\end{equation}
which is the particular case of the Robin problem when $\theta = 0$ and we take $f_{u} = \frac{\partial u_{f}}{\partial x_{d}}|_{x_{d}=0}$, $g_{u} = N_{q}f_u = u_{f}|_{x_d =0}$.

\end{proof}

\begin{remark}
Although this property of the mappings is not necessary for our proof, we find it interesting and it can be useful for some other problems.
\end{remark}

\subsection{Uniqueness for a wave propagation inverse problem in a half-plane}\label{clu2d}

The aim of this section is to determine $q(x)$ in the upper half-plane from knowledge of the Dirichlet-to-Neumann map $\Lambda_{q}$.  For this purpose, we follow the same arguments as in \cite{clu}, adapting them to the two dimensional case.

The main idea is to first prove that if $\Lambda_{q_1}=\Lambda_{q_2}$ on some open subset of the boundary $x_2 = 0$, then the following orthogonality relation for the potentials $q_1$ and $q_2$ holds
\begin{equation}\label{8}
\int_B (q_1 - q_2)v_1 v_2 \, dx = 0.
\end{equation}
Here $B$ is an open set containing the supports of $q_1-k^2$ and $q_2-k^2$, while $v_1$ and $v_2$ are solutions of the Dirichlet problem

\begin{equation}\label{di2d}
\left\{ \begin{array}{ll}
(\Delta + q_i)v_i= 0 & \text{in $\R^2_{+}$}\\
\\
v_i|_{x_2 =0} = g & \text{over $\{ x_{2} = 0\}$}\\
\\
v_i \, \, \text{satisfies the outgoing radiation condition},
\end{array} \right.
\end{equation}
for $i=1,2$ and $g\in H^{1/2,-\delta}(\R^1)$.

In order to get (\ref{8}), we will use the associated Dirichlet Green's function, which can be defined by the method of images. We use a tilde to denote the image point, so that $\tilde{y}$ is the point obtained reflecting $y$ across the $y_2 = 0$ plane. First we recall that the free-plane outgoing Green's function corresponding to the medium parameter in the lower half-plane is
\begin{equation}
G(x,y) = \frac{i}{4}H_{0}^{(1)}(k|x-y|), 
\end{equation}
where $H_{0}^{(1)}$ denotes the zeroth order Hankel function of the first kind. Then the Green's function of the homogeneous half-plane is defined by the method of images as
\begin{equation}
G^{D}(x, y) = G(x,y) - G(x,\tilde{y}).
\end{equation} 

This function satisfies the problem
\begin{equation}
\left\{ \begin{array}{ll}
(\Delta + k^{2})G^{D}(x, y)= -\delta(x-y) & \text{in $\R^2_{+}$}\\
\\
G^{D}(x,y) = 0 & \text{over $\{x_{2} = 0\}$}\\
\\
\lim_{|x| \to \infty} |x|\left(\frac{\partial}{\partial |x|} -ik \right) G^{D}(x,y) = 0.
\end{array} \right.
\end{equation}

Thus we can define the perturbed free-plane Green's function associated to the problem (\ref{di2d}) as the unique solution of the Lippman-Schwinger equation
\begin{equation}\label{dir}
G^{D}_q(x,y) = G^{D}(x,y) + \int G^{D}(x,z)(q(z)-k²)G^{D}_q(z,y)\, dz.
\end{equation}

Next we compute the kernel of the Dirichlet-to-Neumann map $\Lambda_{q}$ using this Green's function. More concretely, it can be proved that the kernel of $\Lambda_q$ is
$$
-\frac{\partial}{\partial x_2}\frac{\partial}{\partial y_2} G^{D}_q(x,y)\Big{|}_{x_2 =0, y_2 =0}.
$$
See Proposition 2.1 of \cite{clu} for more details. Therefore, we say that $\Lambda_{q_1} = \Lambda_{q_2}$ on some open subset $\Gamma$ of the boundary $x_2 =0$ if the kernels of the operators coincide on $\Gamma \times \Gamma$, i.e.
\begin{equation}
\frac{\partial}{\partial x_2}\frac{\partial}{\partial y_2}G_{q_1}^{D}(x,y)\Big{|}_{x_2=0, y_2=0} = \frac{\partial}{\partial x_2}\frac{\partial}{y_2}G_{q_2}^{D}(x,y)\Big{|}_{x_2=0, y_2=0}
\end{equation}
for $x$ and $y$ in $\Gamma$.

\begin{remark}
We point out that one can also prove that the boundary value problem (\ref{di2d}) has a unique solution by using the Dirichlet Green's function $G^D_q$ given above, as done by Cheney and Isaacson \cite{ci} for the three-dimensional case. 
\end{remark}

Now we are ready to state the main result of this section.

\begin{thm}
Suppose the set $B$ containing the supports of $q_1 - k^2$ and $q_2 - k^2$ are strictly contained in the upper half-plane. Let $q_1, q_2 \in L^\infty(B)$. If $\Lambda_{q_1} = \Lambda_{q_2}$ on some open subset $\Gamma$ of the boundary $x_2 = 0$, then $q_1 = q_2$.
\end{thm}

\begin{proof}
The proof will be divided into two steps. Following the same method as in \cite{clu}, it can be proved that under the hypotheses of the theorem, the following orthogonality relation holds
\begin{equation}\label{8'}
\int_{B} (q_1 - q_2)v_1 v_2 \, dx =0
\end{equation}
for $v_1, v_2 \in V$, where $V= \{ v \in H^2(B) : (\Delta + q)v = 0\}$.

We only give the main ideas of the proof. For more details we refer the reader to \cite{clu}. The proof involves a series of five lemmas. The first two show that knowledge of the Dirichlet-to-Neumann map suffices to determine the Dirichlet Green's function outside the perturbation of $q$. The third and fourth lemmas stablish a version of the Green's theorem identity that is often used for uniqueness arguments. The last lemma shows that the linear combinations of the Dirichlet Green's function can be used to approximate the Bukhgeim solutions.

In the three dimensional case, Cheney, Lassas and Uhlmann \cite{clu} use the Sylvester-Uhlmann solutions to conclude from (\ref{8'}) that $q_1=q_2$. In our case, in the two dimensional case, we make use of Imanuvilov-Yamamoto's solutions \cite{iy}. We include here a brief outline of these solutions. 

Let $i=\sqrt{-1}$, $x=(x_1,x_2)$, $x_1, x_2 \in \R$, $z=x_1 + ix_2$, $\bar{z}$ denote the complex conjugate of $z\in\mathbb{C}$. We set $\partial_z=\frac{1}{2}(\partial_{x_1}-i\partial_{x_2})$, $\partial_{\bar{z}}=\frac{1}{2}(\partial_{x_1}+i\partial_{x_2})$ and we introduce the operators 
$$
\partial_{\bar{z}}^{-1}g = -\frac{1}{\pi}\int_{B} \frac{g(\xi_1,\xi_2)}{\zeta-z}\, d\xi_1 \, d\xi_2, \quad \quad \partial_{z}^{-1}g = -\frac{1}{\pi}\int_{B}\frac{g(\xi_1,\xi_2)}{\bar{\zeta}-\bar{z}} \, d\xi_1 \, d\xi_2,
$$
where $\zeta=\xi_1 + i\xi_2$. Consider a holomorphic function $\Phi(x,y)=\left(z-(y_1 + iy_2) \right)^2$ with $y=y_1 +iy_2$ and we introduce two operators
$$
\tilde{R}_{\tau}g = \frac{1}{2}e^{\tau(\bar{\Phi}-\Phi)}\partial_{z}^{-1}\left(ge^{\tau(\Phi-\bar{\Phi})} \right), \quad \quad R_\tau g = \frac{1}{2}e^{\tau(\Phi-\bar{\Phi})}\partial_{\bar{z}}^{-1}
\left(ge^{\tau(\bar{\Phi}-\Phi)} \right).
$$
Set $U_0 = 1$, $U_1 = \tilde{R}_\tau \left(\frac{1}{2}\left(\partial_{\bar{z}}^{-1}q_1 -\beta_1 \right) \right)$, $U_j = \tilde{R}_\tau\left(\frac{1}{2}\partial_{\bar{z}}^{-1}(q_1 U_{j-1}) \right)$ for all $j\geq 2$, where $\beta_1$ is a fixed constant. We then construct a solution to the Schr\"odinger equation with $q_1$ in the form
\begin{equation}\label{Aiy}
v_1 = \sum_{j=0}^\infty e^{\tau\Phi}(-1)^j U_j.
\end{equation}
Similarly, we also construct the complex geometric optics solution for the Schr\"odinger equation with the potential $q_2$ 
\begin{equation}\label{Biy}
v_2 = \sum_{j=0}^\infty e^{-\tau\bar{\Phi}} (-1)^j V_j,
\end{equation}
where $V_0 =1$, $V_1=R_{-\tau}(\partial_z^{-1}q_2 - \beta_2)$, $V_j = R_{-\tau}\left(\partial_z^{-1}(q_2 V_{j-1}) \right)$ and $\beta_2$ is a  specific constant. 

After plugging the solutions (\ref{Aiy}) and (\ref{Biy}) into the orthogonality relation (\ref{8'}), by the stationary phase argument, it can be showed that
\begin{equation}\label{Ciy}
0 = \int_B (q_1 - q_2) v_1 v_2 \, dx = \int_B (q_1 - q_2)e^{\tau(\bar{\Phi} - \Phi)}\, dx + o\left(\frac{1}{\tau} \right) \quad \quad \text{as} \quad \tau \to +\infty.
\end{equation}
Finally, from (\ref{Ciy}) by using some measure theory and the stationary phase method, it can be concluded that $q_1 - q_2 = 0$.

\end{proof}

\section{Appendix: Unperturbed Robin Green's function}\label{appendix}

Let $d=2, 3$. In this section we present the Green's function satisfying the boundary value problem
\begin{equation}\label{ro}
\left\{ \begin{array}{ll}
(\Delta + k^{2})G_{\theta}(x,y) = -\delta(x-y) & \text{in $\R^d_{+}$}\\
\\
\frac{\partial G_{\theta}(x,y)}{\partial x_{d}} + \theta G_{\theta}(x,y) = 0 & \text{over $\{x_{d} = 0\}$},
\end{array} \right.
\end{equation}
where $y\in \Rd$ is a fixed source point and $x\in\Rd$ is the receiver point.

A useful method for constructing the Robin Green's function for a general $\theta$ is taking a Fourier transform in the horizontal directions $x_1$, $x_{d-1}$ and reducing the PDE of the problem (\ref{ro}) to an ordinary differential equation. Thus the spectral Green's function can be given by
\begin{equation}\label{fouriergreen}
\widehat{G_\theta}(\xi,k,x_d,y_d)= C_\pi\left(\frac{\theta +\sqrt{\xi^2 - k^2}}{\theta-\sqrt{\xi^2 -k^2}}\frac{e^{-\sqrt{\xi^2 -k^2}(x_d + y_d)}}{\sqrt{\xi^2 - k^2}} - \frac{e^{-\sqrt{\xi^2 -k^2}|x_d -y_d|}}{\sqrt{\xi^2 - k^2}} \right)
\end{equation}
where $C_\pi= \frac{1}{\sqrt{8\pi}}$ when $d=2$, $C_\pi = \frac{1}{4\pi}$ if $d=3$. Using the inverse Fourier transform, the spatial Green's function is represented as
\begin{equation}\label{spatialrobin}
G_{\theta}(x,y)=C'_\pi\int_{\R^{d-1}} \widehat{G_{\theta}}(\xi,k,x_d,y_d)e^{-i\xi\cdot(x'-y')} \, d\xi,
\end{equation}
for $C'_\pi = \frac{1}{4\pi}$ in the two dimensional case and $C'_\pi=\frac{1}{2\pi}$ when $d=3$. See \cite{DMN2} and \cite{DMN1} for more details.

We consider the cases when $\theta >0$, the non-absorbing boundary case, studied by Dur\'an, Muga and N\'ed\'elec \cite{DMN2} ($d=2$), \cite{DMN1}-\cite{DMN} (d=3) and also the rigid or energy-absorbing boundary case, $\theta =0$ or $\Im \theta > 0$, given by Thomasson \cite{Th}, Chandler-Wilde \cite{cw} ($d=3$) and Chandler-Wilde, Hothersall \cite{cw1}-\cite{cw2} ($d=2$). We include here a brief summary of these works pointing out the main properties of $G_{\theta}$ that we need for solving the direct problem. More precisley, we focus on the asymptotic expansion of $G_\theta$.

\subsection{Non-absorbing boundary ($\theta > 0$)}\label{sectiongreen}

The Green's function is given by (\ref{spatialrobin}). In order to get its asymptotics, some analysis techniques like the stationary phase and the calculus of residues are used. 

In the two dimensional case, it has been showed \cite{DMN2} that
$$
G_{\theta}(x,y) = G_{\theta}^1 (x,y) + G_{\theta}^2(x,y),
$$
where
\begin{equation}\label{grb20}
G_{\theta}^1(x,y) = \left\{ \begin{array}{ll}
\left(\frac{\theta - ik \sin\gamma}{\theta + ik\sin\gamma}e^{2ik\sin\gamma y_2} -1 \right)\frac{e^{i(kr +\frac{\pi}{4})}}{\sqrt{8\pi kr}} + o\left(r^{-\frac{3}{2}} \right), & \text{when $x_2 - y_2 >0$,}\\
\\
o\left(r^{-1} \right) & \text{when $x_2 - y_2 \leq 0$}
\end{array} \right.
\end{equation}
and
\begin{equation}\label{grb21}
G_{\theta}^2(x,y) = \frac{-i\theta}{\sqrt{\theta^2 + k^2}} e^{-\theta(y_2 + x_2)}e^{i\sqrt{\theta^2 + k^2}|y_1 - x_1|} + o\left(r^{-1}\right).
\end{equation}
Here $(r, \gamma)$ are the polar coordinates for the spatial variables. 
Then, the general radiation condition for $r$ large and $0 < \alpha < \frac{1}{2}$ is the following:
\begin{equation}\label{RCapp}
\left\{ \begin{array}{ll}
\left| \frac{\partial G_{\theta}}{\partial r} - ik G_{\theta}\right| < cr^{-(1-\alpha)} & \text{in $\R^2_{+}(\alpha_{+}) = \{ (x_1, x_2)\in \R^2_{+} : x_{2} > cr^{\alpha}\}$}\\
\\
\left|\frac{\partial G_{\theta}}{\partial r} - i \sqrt{k^{2} + \theta^{2}}G_{\theta} \right| < cr^{-(1-\alpha)} & \text{in $\R^2_{+}(\alpha_{-}) = \{(x_1, x_2)\in \R^2_{+} : x_{2} < cr^\alpha \}$}.
\end{array} \right.
\end{equation}

Regarding to the three dimensional case (\cite{DMN1}, \cite{DMN}), the following estimation is obtained. For $0 < \alpha <1$,
\begin{equation}\label{grb}
G_{\theta}(x,y) = \left\{ \begin{array}{ll}
\left(\frac{\theta - ik \cos\gamma}{\theta + ik\cos\gamma}e^{2ik\cos\gamma y_3} -1 \right)\frac{e^{ikr}}{4\pi r} + O\left(r^{-(2\alpha +1/2)} \right), & \text{when $r\cos\gamma > r^\alpha$,}\\
\\
 \frac{\theta e^{-\theta(x_3 + y_3)}}{i\sqrt{2\pi}(\theta^2 + k^2)^{1/4}}\frac{e^{i\left(\sqrt{\theta^2 + k^2}\rho - \pi/4 \right)}}{\sqrt{\rho}} + O\left(\rho^{-3/2} \right) & \text{when $r\cos\gamma < r^\alpha$},
\end{array} \right.
\end{equation}
as $r\to \infty$ and $\rho \to +\infty$, respectively. Here $(r,\gamma,\varphi)$ are the spherical coordinates in the spatial variables
and $\rho$ stands for the horizontal radial variable defined as the radius of the projection over the horizontal plane, that is,
$$
\rho = r\sin\gamma = \sqrt{(x_1 - y_1)^2 + (x_2 - y_2)^2} = \vert x' - y' \vert,
$$
where $x'=(x_1,x_2)$ and $y'=(y_1,y_2)$.
\noindent
\\
As a consequence, it may be concluded that the Green's function satisfies the following general radiation condition
\begin{equation}\label{RCapp}
\left\{ \begin{array}{ll}
\left| \frac{\partial G_{\theta}}{\partial r} - ik G_{\theta}\right| < \frac{C}{r^{\left(2\alpha + \frac{1}{2} \right)}} & \text{in $\R^3_{+}(\alpha_{+}) = \{x_{3} > r^{\alpha}\}$}\\
\\
\left|\frac{\partial G_{\theta}}{\partial r} - i \sqrt{k^{2} + \theta^{2}}G_{\theta} \right| < \frac{C}{r^{\left(\frac{3}{2} -\alpha \right)}} & \text{in $\R^3_{+}(\alpha_{-}) = \{ 0 \leq  x_{3} < r^\alpha \}$},
\end{array} \right.
\end{equation}
when $r \to + \infty$ and for any $\alpha \in \left(\frac{1}{4}, \frac{1}{2} \right)$.

\subsection{Rigid or energy-absorbing boundary ($\theta= 0$ or $\Im \theta > 0$)}\label{cw}

In the case of a rigid boundary, the solution of the problem (\ref{ro}) is easily found by the method of images to be
\begin{equation}\label{grbcw0}
G_{0}(x,y) = \left\{ \begin{array}{ll}
-\frac{e^{ikR}}{4\pi R} - \frac{e^{ikR'}}{4\pi R'}, & \text{when $d=3$,}\\
\\
  -\frac{i}{4}H_{0}^{(1)}(kR) - \frac{i}{4}H_{0}^{(1)}(kR') & \text{when $d=2$},
\end{array} \right.
\end{equation}
where $R=|x-y|$ and $R'=|x-y'|$, as $y'$ is the image position with $(y_1, y_2, -y_3)$ in the three-dimensional case and $(y_1,-y_2)$ in the two-dimensional one.  In addition, we denote $(r,\gamma)$ the polar coordinates of $x$ and $\rho=kR'$. Then it follows that for $d=2,3$, $G_{0}$ satisfies the usual Sommerfeld radiaton condition
\begin{equation}\label{usuradi}
\frac{\partial G_{0}}{\partial r} - ik G_{0} = o(r^{-\frac{(d-1)}{2}}), \quad \quad \text{as} \quad \quad r\to \infty.
\end{equation}

In the general case $\Im \theta > 0$, $G_{\theta}$ is expressed as the sum of $G_{0}$ and a correction term $P_{\theta}$, i.e.,
\begin{equation}\label{new}
G_{\theta}(x,y) = G_0(x,y) + P_{\theta}(x,y).
\end{equation}
The representation for $P_{\theta}$ and its first derivatives are derived in the form of Laplace-type integral. This allows to obtain asymptotic expansions of $P_{\theta}$ in the far fiel (as $\rho \to \infty$) by using the modified saddle point method of Ott. $P_{\theta}$ is approximated by $P_{\theta,N}$ for $N=0,1, \ldots$, and the following bound on the error of this approximation is obtained \cite{Th}, \cite{cw} \cite{cw1}, \cite{cw2} for $N=0,1,\ldots$ and for all $\rho_{0} >0$:
\begin{equation}\label{error}
|P_{\theta} - P_{\theta,N}| \leq \left\{ \begin{array}{ll}
\frac{C}{\rho^{N+1}}, & \text{when $d=3$,}\\
\\
\frac{C}{\rho^{N+\frac{3}{2}}} & \text{when $d=2$},
\end{array} \right.
\end{equation}
when $\rho \geq \rho_{0}$, where $C$ is a positive constant depending on $N$ and $\rho_0$. Hence,
\begin{equation}
P_{\theta} \to 0 \quad \quad \text{as} \quad \quad \rho \to \infty,
\end{equation}
uniformly on $\theta$. As a consequence,
\begin{equation}\label{grbcw1}
G_{\theta} \to G_0 \quad \quad \text{as} \quad \quad \rho \to \infty,
\end{equation}
and it may be concluded that $G_{\theta}$ also satisfies the radiation condition (\ref{usuradi}).


\begin{thebibliography}{99}

\bibitem{AS} M. Abramowitz, I.A. Stegun, \emph{Handbook of Mathematical Functions with Formulas, Graphs, and Mathematical Tables}, Dover, New York (1965).


\bibitem{A} S. Agmon, \emph{Spectral properties of Schr\"odinger operators and scattering theory}, Ann. Scuola Norm. Sup. Pisa Cl. Sci. 4 (1975), 151-218.


\bibitem{Bl1} E. Bl\aa sten, \emph{Stability and uniqueness for the inverse problem of the Schr\"odinger equation in $2D$ with potentials in $W^{\varepsilon,p}$}, \url{arXiv:1106.0632} (2011).

\bibitem{Bu} A.L. Bukhgeim, \emph{Recovering a potential from Cauchy data in the two-dimensional cas}, J. Inv. Ill-Posed Probles, 16 (2008), 19-33.

\bibitem{cw} S.N. Chandler-Wilde, \emph{Ground effects in environmental sound propagation}, Ph.D. Thesis, University of Bradford, U.K. (1988).

\bibitem{cw0} S.N. Chandler-Wilde, \emph{The impedance boundary value problem for the Helmholtz equation in a half-plane}, Math. Methods Appl. Sci. 20, 813-840 (1997).

\bibitem{cw3} S.N. Chandler-Wilde, D.C. Hothersall, \emph{Sound propagation above an inhomogeneous impedance plane}, J. Sound Vibration 98 (1985) 475-491.

\bibitem{cw4} S.N. Chandler-Wilde, D.C. Hothersall, \emph{The boundary integral equation method in outdoor sound propagation}, Proceedings of the Institute of Acoustics 9 (1987) 37-44.

\bibitem{cw1} S.N. Chandler-Wilde, D.C. Hothersall, \emph{Efficient calculation of the Green function for acoustic propagation above a homogeneous impedance plane}, J. Sound Vibration 180 (1995), no. 5, 705-724.

\bibitem{cw2} S.N. Chandler-Wilde, D.C. Hothersall, \emph{A uniformly valid far field asymptotic expansion of the Green function for two-dimensional propagation above a homogeneous impedance plane}, J. Sound Vibration 182 (1995), no. 5, 665-675.





\bibitem{cwp0} S.N. Chandler-Wilde, A.T. Peplow, \emph{A boundary integral equation formulation for the Helmholtz equation in a locally perturbed half-plane}, Z. Angew. Math. Mech. 85 (2005), no. 2, 79-88.

\bibitem{clu} M. Cheney, M. Lassas, G. Uhlmann, \emph{Uniqueness for a wave propagation inverse problem in a half-space}, Inverse Problems 14 (1998) 679-684.

\bibitem{ci} M. Cheney, D. Isaacson, \emph{Inverse problems for a perturbed dissipative half-space}, Inverse Problems 11 (1995) 865-888.

\bibitem{cs1} C.F. Chien, W.W. Soroka, \emph{Sound propagation along an impedance plane}, J. Sound Vribation 43 (1975) 9-20.

\bibitem{cs2} C.F. Chien, W.W. Soroka, \emph{A note on calculation of sound propagation along an impedance surface}, J. Sound Vibration 69 (1980) 340-343.

\bibitem{DMN2} M. Dur\'an, I. Muga, J.C. N\'ed\'elec, \emph{The Helmholtz equation with impedance in a half-plane}, C. R. Acad. Sci. Paris. Ser. I 340 (2005) 483-488.

\bibitem{DMN1} M. Dur\'an, I. Muga, J.C. N\'ed\'elec, \emph{The Helmholtz equation with impedance in a half-space}, C. R. Acad. Sci. Paris. Ser. I 341 (2005) 561-566.

\bibitem{DMN} M. Dur\'an, I. Muga,  J.C. N\'ed\'elec, \emph{The Helmholtz equation in a locally perturbed half-space with non-absorbing boundary}. Arch. Rational Mech. Anal. 191 (2009) 143-172.

\bibitem{DHN} M. Dur\'an, R. Hein, J.C. N\'ed\'elec, \emph{Computing numerically the Green's function of the half-plane Helmholtz operator with impedance boundary conditions}, Numer. Math. 107 (2007) 295-314.

\bibitem{gm1} F. Gesztesy, M. Mitrea, \emph{Generalized Robin boundary conditions, Robin-to-Dirichlet maps, and Krein-type resolvent formulas for Schr\"odinger operators on bounded Lipschitz domains}, Perspectives in partial differential equations, harmonic analysis and applications, 105-173, Proc. Sympos. Pure Math., 79, Amer. Math. Soc., Providence, RI, 2008.

\bibitem{gm2} F. Gesztesy, M. Mitrea, \emph{Robin-to-Robin maps and Krein-type resolvent formulas for Schr\"odinger operators on bounded Lipschitz domains}, Modern analysis and applications. The Mark Krein Centenary Conference. Vol. 2: Differential operators and mechanics, 81-113, Oper. Theory Adv. Appl., 191, Birkh\"auser Verlag, Basel, 2009.


\bibitem{Ha} D. Habault, \emph{Sound propagation above an inhomogeneous plane: boundary integral equation methods}, J. Sound Vibration 100 (1985) 55-67.

\bibitem{hf} D. Habault, P.J.T. Filippi, \emph{Ground effect analysis: surface wave and layer potential representations}, J. Sound Vibration 79 (1981) 529-550.


\bibitem{ht} L. Hwang, E. Tuck, \emph{On the oscillations of harbours of arbitrary shape}, J. Fluid. Mech. 42 (1970) 447-464.

\bibitem{iy} O. Imanuvilov, M. Yamamoto, \emph{Inverse boundary value problem for Schr\"odinger equation in two dimensions}, \url{arXiv:1208.3775} (2012).

\bibitem{in1} M.I. Isaev, R.G. Novikov, \emph{Stability estimates for determination of potential from the impedance boundary map}, \url{arXiv:1112.3728} (2012).

\bibitem{in2} M.I. Isaev, R.G. Novikov, \emph{Reconstruction of a potential from the impedance boundary map}, \url{arXiv:1204.0076} (2012).

\bibitem{Ka1} G. Karamyan, \emph{The inverse scattering problem for the acoustic equation in a half-space}, Inverse Problems 18 (2002) 1673-1686.

\bibitem{Ka} G. Karamyan, \emph{Inverse scattering in a half-space with passive boundary}, Commun. Partial Diff. Equ. 28 (9-10), 1627-1641 (2003).

\bibitem{Ka2} G. Karamyan, \emph{The inverse scattering problem with impedance boundary in a half-space}, Inverse Problems 20 (2004) 1485-1495.

\bibitem{khn} T. Kawai, T. Hidaka, T. Nakajima, \emph{Sound propagation above an impedance boundary}, J. Sound Vibration 83 (1982) 125-138.

\bibitem{ma} C. L. Mader, \emph{Numerical Modeling of Detonation}, University of California Press, Berkeley (1979).

\bibitem{Ott} H. Ott, \emph{Die Sattelpunktsmethode in der Umgebung eines Pols}, Annalen der Physik 43 (1943) 393-403.

\bibitem{Ra} K.B. Rasmussen, \emph{The effect of terrain profile on sound propagation outdoors}, Report Number 111, Danish Acoustic Institute (1982).

\bibitem{r} A.D. Rawlins, \emph{The field of a spherical wave reflected from a plane absorbent surface expressed in terms of an infinite series of Legrende polynomials} J. Sound Vibration 89 (1983) 359-363.

\bibitem{sa} K. Sakoda, \emph{Optical Properties of Photonic crystals}, Springer Series in Optical Sciences 80. Springer, Berlin (2001).

\bibitem{Sch} G. Schwarz, \emph{Hodge decomposition - A method for solving boundary value problems}, Springer 1995.


\bibitem{su} J. Sylvester, G. Uhlmann, \emph{A global uniqueness theorem for an inverse boundary value problem}, Ann. of Math., 125 (1987), 153-169.

\bibitem{Th} S.-I. Thomasson, \emph{A powerful asymptotic solution for sound propagation above an impedance boundary}, Acustica 45 (1980) 122-125.

\bibitem{V} I.N. Vekua, \emph{Generalized Analytic Functions}, Oxford, Pergamon Press, 1962.



\end{thebibliography}
\end{document}